\documentclass[11pt,a4paper]{article}
\usepackage{amsfonts, amsmath, amssymb,  amsthm, color}
\usepackage{latexsym}
\title{
Leray weak solutions of the\\ Incompressible Navier Stokes system on exterior domains via the artificial compressibility method
}
\author{\textbf{Donatella Donatelli and Pierangelo Marcati}\\
        {\small Dipartimento di Matematica Pura ed Applicata}\\
       {\small Universit\`a degli Studi dell'Aquila}\\
       {\small 67100 L'Aquila, Italy}\\
        $\scriptstyle\mathtt{\{donatell, marcati@univaq.it\}}$\\}
\date{}
\newcommand{\e}{\varepsilon}		       
\newcommand{\R}{\mathbb{R}}

\newcommand{\ue}{u^{\varepsilon}}
\newcommand{\ut}{\tilde{u}}
\newcommand{\pe}{p^{\varepsilon}}
\newcommand{\pt}{\tilde{p}}
\newcommand{\dive}{\mathop{\mathrm {div}}}

\newtheorem{theorem}{Theorem}[section]   
\newtheorem{corollary}[theorem]{Corollary}
\newtheorem{lemma}[theorem]{Lemma}
\newtheorem{proposition}[theorem]{Proposition}

\theoremstyle{definition}
\newtheorem{definition}[theorem]{Definition}
\newtheorem{remark}[theorem]{Remark}

\numberwithin{equation}{section}

\begin{document}

\maketitle
\begin{abstract}
In this paper we study  the Leray weak solutions of the incompressible Navier Stokes equation in an exterior domain.We describe, in particular,  an hyperbolic version of the so called artificial compressibility method investigated by J.L.Lions and Temam. The convergence of these type of approximation show in general  a lack of strong convergence due to the presence of acoustic waves. In this paper we face this difficulty  by taking care of the dispersive nature of these waves  by means of the Strichartz estimates or waves equations satisfied by the pressure. We actually decompose the pressure in different acoustic components, each one of them satisfies
a specific initial boundary value problem. The strong convergence analysis of the velocity field will be achieved by using the associated Leray-Hodge decomposition. 
\medbreak 
\medbreak
\end{abstract}
\section{Introduction}

This paper is concerned with the description of a type of approximation method for the weak solutions of the $3-D$ Navier Stokes equations over an exterior domain $\Omega$. We say that $\Omega$ is an exterior domain if it is the complement in $\R^{3}$ of a compact set  (usually called compact obstacle). In order to solve the Navier Stokes equations in an exterior domain we need to find the velocity and the pressure fields which together solve the equations and which moreover assume given boundary data on the obstacle. This latter requirement contains usually the most relevant difficulty in this type of analysis.\\
The most simple examples of such kind of flow are fluids  filling up space and flowing past spheres, plates and cylinders as the flow of a river around the stones lying on the riverbed. The interest in studying the Navier Stokes equations in such kind of domains  arises from many phenomena in physics and  applications in engineering models (see \cite{WH07} and references therein). More complicated examples are the motion of bubbles in a liquid ( we may think at the bubbles in the ocean) and the sedimentation of particles. In these cases it is important to determine the forces that the fluid exerts on the structures. Similar phenomena comes also from climate modeling as for the rain drops falling within clouds in the high atmosphere and for these reasons even the flow around the sphere, as the simples example of a falling particle, is still a subject of great interest \cite{N04}. Other examples can be included in the same framework from engineering models like the design of project  aircrafts wings with a high speed airfoil. Interesting information concerning these questions can also be obtained by studying insect flight, \cite{WZ00}. Indeed,  if we consider the air as a fluid the insect can be seen as a moving obstacle. These kind of measurements are considered important and have many applications in construction micro-air vehicles for reconnaissance mission and space missions as the flight in the atmosphere of Mars, \cite{MN03}. A final example comes from the area of medicine  in the study of hemodynamics where it could be of great importance to understand and model the blood flow around an embolus or a prothesis.\\
The mathematical model of an incompressible fluid in an exterior domain is given by the following set of equations
\begin{equation}
    \begin{cases}
	\partial_{t}u+\dive(u\otimes u)-\mu\Delta u=\nabla p+f & x\in\Omega,\ t\geq 0\\
	\dive u=0 & x\in\Omega,\ t\geq 0,\\
	u(x,0)=u_{0}(x) &x\in\Omega,\\
	u|_{\partial\Omega}=0, \quad \lim_{|x|\to\infty}u=0, 
	\end{cases}
    \label{3.1}
\end{equation}
where $\Omega$ is  an exterior domain of $\R^{3}$, $u\in\R^{3}$ denotes the velocity vector field , $p\in \R$ the pressure of the fluid , $f\in \R^{3}$ is a given external force, $\mu$ is the kinematic viscosity.
In the case of  the whole space $\R^{d}$ or of a bounded domain there exists, in the mathematical literature, several results concerning the existence and regularity of Leray weak solutions to the Navier Stokes equations, for example we can refer to books of P.L.Lions \cite{LPL96} and Temam \cite{Tem01}. The exterior problem  for the Navier Stokes equation consists of finding in the region exterior to a closed bounded surface, velocity and pressure functions which together solve the equations and are such that the velocity assumes given values on the surface and tend to a prescribed limit at infinity. For what concerns the existence theory of weak solutions  in the case of an exterior domain a first result can be found in the paper of Leray \cite{Ler34}. We recall here what we mean for a Leray weak solution of the system \eqref{3.1}.
\begin{definition}
We say that $u$ is a Leray weak solution of the Navier Stokes equation if it 
satisfies  \eqref{3.1} in the sense of distributions, namely
\begin{align*}
&\int_{0}^{T}\!\!\int_{\Omega}\left(\nabla u\cdot\nabla\varphi -u_{i}u_{j}\partial_{i}\varphi_{j}-u\cdot
\frac{\partial \varphi}{\partial t}\right) dxdt\\
&=\int_{0}^{T}\langle f ,\varphi \rangle_{H^{-1}\times H^{1}_{0}} dxdt+\int_{\Omega}u_{0}\cdot \varphi dx,
\end{align*}
for all $\varphi\in C^{\infty}_{0}(\Omega\times[0,T])$, $\dive \varphi =0$ and
\begin{equation*}
\dive u=0 \qquad \text{in $\mathcal{D}'(\Omega\times[0,T])$}
\end{equation*}
and the following energy inequality holds
\begin{align*}
\frac{1}{2}\int_{\Omega}&|u(x,t)|^{2}dx+\mu\int_{0}^{t}\!\!\int_{\Omega}|\nabla u(x,t)|^{2}dxds\\ \leq
&\frac{1}{2}\int_{\Omega}|u_{0}|^{2}dx +\int_{0}^{t}\langle f ,u \rangle_{H^{-1}\times H^{1}_{0}}ds,\qquad \text{for all $t\geq 0$}.
\end{align*}
\end{definition}
Similarly to the problems on the whole space or on bounded domains, also here, there is a large amount of literature concerning the  regularity of solutions. For instance, Finn in \cite{Fi59a}, \cite{Fi59b}, \cite{Fi61}, \cite{Fi65} studied the exterior stationary problem  within the class of solutions, called by him as ``physically reasonable'',  which tend to a limit at infinity like $|x|^{-\frac{1}{2}-\eta}$ for some $\eta>0$. For small data he proved both existence and uniqueness in this class. In the case of nonstationary flow the question of existence and uniqueness was later addressed by Heywood. In \cite{He72} he recovered stationary solutions as the limit of nonstationary solutions. Then, he also studied the stability and regularity of  nonstationary solutions in \cite{He74a}, \cite{He74b}, \cite{He79}, \cite{He80}. Further regularity properties of solutions in exterior domain can be found in \cite{GaM86a}, \cite{GaM86b}, \cite{Ma87}, \cite{MS88}, \cite{KY98}. In \cite{GH07} the theory of the Navier Stokes equations was developed considering a moving or a rotating obstacle.
 
Motivated by the previous examples and applications there have been also considerable efforts to develop numerical approximation methods. One of the major difficulty regards the development of numerical schemes including in an efficient way the incompressibility constraints. Chorin \cite{Ch68}, \cite{Ch69}, Temam \cite{Tem69a}, \cite{Tem69b} and Oskolkov \cite{Osk}, in the case of a bounded domain,  to overcome the computational  difficulties connected with   the incompressibility constraints introduced what they named ``artificial compressibility approximation''. They considered a family of perturbed systems, depending on a positive parameter $\e$, which approximate in the limit the Navier Stokes equation and which contain a sort of ``linearized'' compressibility condition, namely the following system
\begin{equation}
\begin{cases}
\displaystyle{\partial_{t}\ue+\nabla \pe=\mu\Delta \ue-\left(\ue\cdot\nabla\right)\ue -\frac{1}{2}(\dive \ue)\ue}+f^{\e}\\
\e\partial_{t}\pe+ \dive \ue=0,
\end{cases}
\label{1.1}
\end{equation}
where $\Omega\in \R^{3}$, $(x,t)\in \Omega\times [0,T]$, $\ue=\ue(x,t)\in\R^{3}$  and $\pe=\pe(x,t)\in \R$, $f^{\e}=f^{\e}(x,t)\in \R^{3}$. \\
The papers of Temam \cite{Tem69a}, \cite{Tem69b} and his book \cite{Tem01} discuss the convergence of these approximations on bounded domains by exploiting the classical Sobolev compactness embedding and they recover compactness in time by the well known J.L. Lions \cite{L-JL59} method of fractional derivatives. 
The same system was used in \cite{DM06} in the case of the Navier Stokes equations in the whole space $\R^{3}$ and modified in a suitable way in \cite{Don08}  for the Navier Stokes Fourier system in $\R^{3}$. In  those papers the authors have carefully to estimate  the acoustic waves for the pressure $\pe$ which is the cause of the lost of the strong convergence. In order to overcome these difficulties they exploit the dispersive properties of these waves, by using  $L^{p}-L^{q}$ estimates of Strichartz type \cite{GV95}, \cite{KT98}, \cite{S77}. 
Here, in order to approximate the system \eqref{3.1} we introduce the system \eqref{1.1} with appropriate boundary data, which makes the problem considerably more difficult as we will see later on. We can see from the second equation of \eqref{1.1} that as $\e$ goes to $0$, the acoustic pressure waves propagate with high speed of order $1/\e$ in the space domain. Because of the fast propagating of the acoustics one expect the velocity $\ue$ to converges only weakly to the incompressible solution of the Navier Stokes equation. Here we overcome this trouble by using the dispersion of these waves at infinity obtaining the strong convergence of $\ue$. In this paper  we need  dispersive estimates of Strichartz type  on  exterior domains.
 These estimates have been recovered by Smith and Sogge, \cite{SmSo95}, \cite{SmSo00} in the case of odd space dimension and by Burq \cite{Burq03} and Metcalfe \cite{Met04} for even space dimension (see Section \ref{secstr}). 
 The connection with the dispersive analysis of the acoustic wave equation has also been considered 
 to study the incompressible limit problem. Similar phenomena appear also in the modeling the Debye screening effect for semiconductor devices,  \cite{DM08}. It is worth to mention here that this type of singular limits from hyperbolic to parabolic systems is not covered and doesn't fits in the general framework of diffusive limits analyzed in \cite{DM04}.
 
In order to understand the  additional difficulty with respect to the whole domain case it is important to remark that the presence of a boundary requires consistent boundary conditions for the approximating problems. This  problem  is not present for the approximating velocity $\ue$ since the limit velocity satisfies Dirichlet boundary condition, hence  we can impose the same condition on $\ue$. Unfortunately for the pressure we have a different situation, indeed for the limit problem (incompressible Navier Stokes) the natural physical boundary condition is of Neumann type, while the wave equation structure for $\pe$ and the Strichartz estimates require Dirichlet type conditions. This problem is well known  also in the numerical literature (see for instance Chorin \cite{Ch67}, Gresho and Sani \cite{GS87}, Sani, Shen and al. \cite{GSSP06} and the references therein). Motivated by the formal analysis of  the previous mentioned papers we introduce here appropriate Dirichlet boundary conditions for the pressure which are expected to be  consistent with the  Neumann boundary condition of the limit problem (see Section 3 for further details). However, to make our analysis rigorous it will be necessary to decompose the approximating pressure in different acoustic components, each one of them with its own appropriate boundary conditions.

This paper is organized as follows. In Section 2 we recall the mathematical tools and basic definitions that we need through the paper and we describe the Strichartz estimate we are going to use. In Section 3 we introduce the approximating system and we state our main result. The Section 4 is devoted to the a priori estimates that are derived from standard energy type estimates. In Section 5 we recover further estimates by exploiting the wave equation structure for the pressure. In Section 6 we show the strong convergence of the approximating sequences. Finally, in Section 7 we prove our main result.

\section{Notations and Preliminaries}

For convenience of the reader we establish some notations and recall some basic theorems that will be useful in the sequel.\\
From now on $\Omega$ denotes an exterior domain to a compact obstacle in $\R^{d}$.  Precisely, $\Omega$ is the complement in $\R^{d}$ to a compact, strictly convex, smooth set contained in $\{|x|\leq R\}$. Moreover, $\Omega$ is assumed to be non trapping in the sense that there is a number $L_{R}$ such that no geodesic of length $L_{R}$ is completely contained in $\{|x|\leq R\}\cap \Omega$.
\subsection{Definition of spaces and Operators}
We will denote by  $\mathcal{D}(\Omega \times \R_+)$
 the space of test function
$C^{\infty}_{0}(\Omega \times \R_+)$, by $\mathcal{D}'(\Omega \times
\R_+)$ the space of Schwartz distributions and $\langle \cdot, \cdot \rangle$
the duality bracket between $\mathcal{D}'$ and $\mathcal{D}$. Moreover
$W^{k,p}(\Omega)=(I-\Delta)^{-\frac{k}{2}}L^{p}(\Omega)$ and $H^{k}(\Omega)=W^{k,2}(\Omega)$ denote the nonhomogeneous Sobolev spaces for any $1\leq p\leq \infty$ and $k\in \R$. The notations
 $L^{p}_{t}L^{q}_{x}$ and $L^{p}_{t}W^{k,q}_{x}$ will abbreviate respectively  the spaces $L^{p}([0,T];L^{q}(\Omega))$ and $L^{p}([0,T];W^{k,q}(\Omega))$. In the definition of homogenous Sobolev spaces we have to be more precise. The homogeneous Sobolev norm $\dot{H}^{\gamma}(\R^{d})$  on the whole space $\R^{d}$ is given by $\|f\|_{\dot{H}^{\gamma}(\R^{d})}=\|(\sqrt -\Delta)^{\gamma}f\|_{L^{2}(\R^{d})}$. On the exterior domain $\Omega$ we have to define the space  $\dot{H}_{D}^{\gamma}(\Omega)$, $\gamma\in\R$ which is the homogeneous Sobolev space associated to the square root of the Laplace operator with Dirichlet boundary condition on $\Omega$, $\sqrt{-\Delta_{D}}$. To be more precise fix $\beta\in C^{\infty}_{0}(\R^{d})$ a smooth cutoff function such that $\beta(x)=1$ for  $|x|\leq R$ and let be $\widetilde{\Omega}$ a compact manifold with the boundary containing $B_{R}=\Omega\cap \{|x|\leq R\} $, we are able to define  
 $$\|f\|_{\dot{H}^{\gamma}_{D}(\Omega)}=\|\beta f\|_{\dot H_{D}^{\gamma}(\widetilde{\Omega})}+\|(1-\beta)f\|_{\dot{H}^{\gamma}(\R^{d})}.$$
 Notice that for functions with support in $\{|x|\leq R\}$ we have
 $$\|f\|_{\dot{H}^{\gamma}_{D}(\Omega)}=\| f\|_{\dot H_{D}^{\gamma}(\widetilde{\Omega})}.$$
 Functions $f\in \dot H_{D}^{\gamma}(\widetilde{\Omega})$ satisfy the Dirichlet conditions $f|_{\partial \widetilde{\Omega}}=0$ and when $\gamma\geq 2$ we must  require the compatibility condition
 $$\Delta^{j}f |_{\partial\widetilde{\Omega}}=0\qquad \text{for $2j<\gamma$}.$$
With the Dirichlet condition fixed we may define the spaces $\dot H^{\gamma}(\widetilde{\Omega})$ in terms of eigenfunctions of $\Delta$. Since $\widetilde{\Omega}$ is compact we have $\{v_{j}\}\subset H^{\gamma}_{D}(\widetilde{\Omega})\cap C^{\infty}(\widetilde{\Omega})$ an orthonormal basis of $L^{2}(\widetilde{\Omega})$ with $\Delta v_{j}=-\lambda_{j}v_{j}$, $\lambda_{j}>0$, $\lambda_{j}\uparrow\infty$. For $\gamma\geq 0$ we define
$$\dot{H}_{D}^{\gamma}(\widetilde{\Omega})=\bigg\{f\in L^{2}(\widetilde{\Omega})\mid \sum_{j\geq 0}|\hat{f}(j)|^{2}\lambda_{j}^{\gamma}<+\infty\bigg\},$$
where $\hat{f}(j)=(f,v_{j})$. The $\dot{H}_{D}^{\gamma}$ norm is given by
$$\|f\|_{\dot{H}^{\gamma}(\widetilde{\Omega})}^{2}=\sum_{j\geq 0}|\hat{f}(j)|^{2}\lambda_{j}^{\gamma}.$$
For $\gamma<0$, we define $\dot{H}^{\gamma}_{D}(\Omega)$ in term of duality. Moreover we mention that  for $r<s$,
$$\|f\|_{\dot{H}^{r}_{D}(\Omega)}\leq C \|f\|_{\dot{H}^{s}_{D}(\Omega)}.$$
For further details see \cite{Burq03} and \cite{Met04}.\\

Let us define in the case of an exterior domain the   Leray's projector $P$ on the space of divergence - free vector fields and $Q$ on the space of gradients vector fields.
Any vector field $\overrightarrow{F}$ on $\Omega$ with $\overrightarrow{F}|_{\partial \Omega}=0$ can be decomposed in the form
$$\overrightarrow{F}=\overrightarrow{K}+\nabla \Lambda,$$
where $\dive \overrightarrow{K}=0$ and $\overrightarrow{K}\cdot n=0$ on $\partial\Omega$ ($n$ is the exterior normal vector to $\partial\Omega$). In fact we can recover $\Lambda$ as a solution of the following system
$$\begin{cases}
\Delta \Lambda=\dive F\\
\displaystyle{\frac{\partial\Lambda}{\partial n}=0\  \text{in $\partial\Omega$},\qquad \int_{\Omega}\Lambda=0},
\end{cases}$$
that is well known has a unique solution (see \cite{CM93}). It turns out that $P$ and $Q$ assume the form 
\begin{equation}
Q=\nabla \Delta^{-1}_{N}\dive\qquad P=I-Q,
\label{2.1.1}
\end{equation} 
where $ \Delta^{-1}_{N}$ is the inverse of the Laplace operator with Neumann boundary condition. It can be proved that $P$ and $Q$ are projections. Let us remark that contrary on what happens in the case of the whole domain, here  $P$ and $Q$ don't commute with translations and so with derivatives, but they are still bounded in $L^{p}$ and so in  $W^{k,p}$ $(1< p<\infty)$ space (see \cite{FKS07} and references therein). 
\subsection{Strichartz estimates in exterior domain}
\label{secstr}
Let us consider the following wave equation defined in the space $[0,T]\times \R^{d}$
\begin{equation*}
\begin{cases}
\left(\partial ^{2}_{t}-\Delta\right)w(t,x)=F(t,x)\\
w(0,\cdot)=f,\quad \partial_{t}w(0,\cdot)=g,
\end{cases}
\end{equation*}
for some data $f,g, F$ and time $0<T<\infty$.
As is well known the wave equation belongs to the so called dispersive equations.  In 1977 Strichartz \cite{S77} realized that combining the dispersive properties of the wave equation with the restriction theorem of the Fourier transform on manifolds he could set up the following estimate
\begin{equation}
\|w\|_{L^{4}_{t}L^{4}_{x}}+\|\partial_{t}w\|_{L^{4}_{t}W^{-1,4}_{x}}\lesssim \|f\|_{\dot H^{1/2}}+\|g\|_{\dot H^{-1/2}}+\|F\|_{L^{4/3}_{t}L^{4/3}_{x}}.
\label{2.2.6}
\end{equation}
Later on this estimate was generalized to the following one (see \cite{GV95}, \cite{KT98})
\begin{equation}
\|w\|_{L^{q}_{t}L^{r}_{x}}+\|\partial_{t}w\|_{L^{q}_{t}W^{-1,r}_{x}}\lesssim \|f\|_{\dot H^{\gamma}_{x}}+\|g\|_{\dot H^{\gamma -1}_{x}}+\|F\|_{L^{\tilde{q}'}_{t}L^{\tilde{r}'}_{x}},
\label{2.2.5}
\end{equation}
where $(q,r)$, $(\tilde{q},\tilde{r})$ have to be  \emph{wave admissible pairs}, namely they satisfy 
\begin{equation}
\left .
\begin{split}
\frac{2}{q}\leq (d-1)&\left(\frac{1}{2}-\frac{1}{r}\right) \qquad 
\frac{2}{\tilde{q}}\leq (d-1)\left(\frac{1}{2}-\frac{1}{\tilde{r}}\right)\\
&\frac{1}{q}+\frac{d}{r}=\frac{d}{2}-\gamma=\frac{1}{\tilde{q}'}+\frac{d}{\tilde{r}'}-2.
\label{2.2.3}
\end{split}\right\}
\end{equation}
The estimate \eqref{2.2.5} still go under the name of Strichartz estimate.
A further generalization of the estimate \eqref{2.2.5} is given when we consider the wave equation on an exterior domain with Dirichlet boundary conditions, namely $w$ is a solution of the following system
\begin{equation*}
\begin{cases}
\left(\partial_{t}^{2}-\Delta\right)w(t,x)=F(t,x) & (t,x)\in \R_{+}\times \Omega\\
w(0,\cdot)=f(x)\in \dot{H}^{\gamma}_{D}\\
\partial_{t}w(0,x)=g(x)\in \dot{H}^{\gamma-1}_{D}\\
w(t,x)=0, &x\in\partial\Omega,
\end{cases}
\end{equation*}
for some data $f,g, F$ and time $0<T<\infty$.
Then, $w$ satifies the following Strichartz estimate, 
\begin{equation}
\|w\|_{L^{q}_{t}L^{r}_{x}}+\|\partial_{t}w\|_{L^{q}_{t}W^{-1,r}_{x}}\lesssim \|f\|_{\dot H^{\gamma}_{D}}+\|g\|_{\dot H^{\gamma -1}_{D}}+\|F\|_{L^{\tilde{q}'}_{t}L^{\tilde{r}'}_{x}},
\label{2.2.1}
\end{equation}
provided that $(q,r)$, $(\tilde{q},\tilde{r})$ are \emph{wave admissible pairs}
in the sense of \eqref{2.2.3}. As we can observe
the estimate \eqref{2.2.1} has the same structure as the one  \eqref{2.2.5} on  the whole space $\R^{d}$, but in order to prove it is necessary to establish new decay estimates. These estimates are different if the space dimension is odd or even because of the lack of strong Huygen's principle in the latter case. The local Strichartz estimate for the homogenous case was  proved  by Smith and Sogge in \cite{SmSo95}. Then in \cite{SmSo00} they established for the nonhomogenous wave equation the global estimate in space and time for odd space dimension. The even space dimension estimate  was obtained independently by Metcalfe \cite{Met04}  and Burq \cite{Burq03}. \\
Of course the estimate \eqref{2.2.1} includes the case of the original  estimate obtained by Strichartz in 1977 \cite{S77}, namely 
\begin{equation}
\|w\|_{L^{4}_{t}L^{4}_{x}}+\|\partial_{t}w\|_{L^{4}_{t}W^{-1,4}_{x}}\lesssim \|f\|_{\dot H^{1/2}_{D}}+\|g\|_{\dot H^{-1/2}_{D}}+\|F\|_{L^{4/3}_{t}L^{4/3}_{x}},
\label{2.2.2}
\end{equation}
From the estimate \eqref{2.2.2} can be deduced the following one that will be usefull for us in the sequel,
\begin{equation}
\|w\|_{L^{4}_{t}L^{4}_{x}}+\|\partial_{t}w\|_{L^{4}_{t}W^{-1,4}_{x}}\lesssim \|f\|_{\dot H^{1/2}_{D}}+\|g\|_{\dot H^{-1/2}_{D}}+\|F\|_{L^{1}_{t}L^{2}_{x}}.
\label{s1}
\end{equation}
Later on we shall also use \eqref{2.2.1} in the case of $d=3$, $(\tilde{q}', {\tilde{r}'})=(1, 3/2)$,  then $\gamma=1/2$ and $(q,r)=(4,4)$, namely the following estimate
\begin{equation}
\|w\|_{L^{4}_{t,x}}+\|\partial_{t}w\|_{L^{4}_{t}W^{-1,4}_{x}}\lesssim \|f\|_{\dot H^{1/2}_{D}}+\|g\|_{\dot H^{ -1/2}_{D}}+\|F\|_{L^{1}_{t}L^{3/2}_{x}}.
\label{s2}
\end{equation}
\subsection{Preliminary theorems and lemma}
Finally we mention here the following technical lemma that will allow us to get $L^{p}$ estimate by means of $W^{-k,p}$ norms.
\begin{lemma}
\label{ly}
Let us consider  a smoothing kernel $j\in C^{\infty}_{0}(\Omega)$, such that $j\geq 0$, $\int_{\R^{d}}j dx=1$, and define the Friedrichs mollifiers as
\begin{equation*}
j_{\alpha}(x)=\alpha^{-d}j\left(\frac{x}{\alpha}\right).
\end{equation*}
Then  for any $f\in \dot H^{1}(\Omega)$, one has
\begin{equation}
\label{y1}
\|f-f\ast j_{\alpha}\|_{L^{p}(\Omega)}\leq C_{p}\alpha^{1-\sigma}\|\nabla f\|_{L^{2}(\Omega)},
\end{equation}
where
\begin{equation*}
p\in [2, \infty)
\quad \text{if $d=2$}, \quad p\in [2, 6] \quad \text{if $d=3$ \ and}\quad \sigma=d\left(\frac{1}{2}-\frac{1}{p}\right).
\end{equation*}
Moreover the following Young type inequality holds
\begin{equation}
\label{y2}
\|f\ast j_{\alpha}\|_{L^{p}(\Omega)}\leq C\alpha^{s-d\left(\frac{1}{q}-\frac{1}{p}\right)}\|f\|_{W^{-s,q}(\Omega)},
\end{equation}
for any $p,q\in [1, \infty]$, $q\leq p$,  $s\geq 0$, $\alpha\in(0,1)$.
\end{lemma}
Finally, we need to  recall  the following compactness tool (see \cite{Si}).
\begin{theorem}
\label{LA}
Let be $\mathcal{F}\subset L^{p}([0,T];B)$,  $1\leq p<\infty$, $B$ a Banach space. $\mathcal{F}$ is relatively compact in  $L^{p}([0,T];B)$ for $1\leq p<\infty$, or in $C([0,T];B)$ for $p=\infty$ if and only if 
\begin{itemize}
\item[{\bf (i)}]
$\displaystyle{\left\{\int_{t_{1}}^{t_{2}}f(t)dt,\ f\in B\right\}}$ is relatively compact in $B$, $0<t_{1}<t_{2}<T$,
\item[{\bf (ii)}]
$\displaystyle{\lim_{h\to 0}\|f(t+h) - f(t)\|_{L^{p}([0, T-h];B)}=0}$ uniformly for any $f \in \mathcal{F}$.
\end{itemize}
\end{theorem}

\section{Approximating system and main result}
As we explained in the Introduction in order to approximate the system \eqref{3.1} we introduce the following system
\begin{equation}
\begin{cases}
\displaystyle{\partial_{t}\ue+\nabla \pe=\mu\Delta \ue-\left(\ue\cdot\nabla\right)\ue -\frac{1}{2}(\dive \ue)\ue}+f^{\e}, \\
\e\partial_{t}\pe+ \dive \ue=0,
\end{cases}
\label{3.2}
\end{equation}
where $(x,t)\in \Omega\times [0,T]$, $\ue=\ue(x,t)\in\R^{3}$  and $\pe=\pe(x,t)\in \R$, $f^{\e}=f^{\e}(x,t)\in \R^{3}$. 
As we can notice the constraint ``$\dive u=0$'' of the system \eqref{3.1} has been replaced by the evolution equation
\begin{equation*}
\partial_{t}\pe=-\frac{1}{\e}\dive\ue,
\end{equation*}
which can be seen as the linearization  around a constant state of the continuity equation in the case of a compressible fluid.
Concerning the first equation of the system \eqref{3.2} we can observe that compared to the equation of the balance of momentum it has the extra term $-1/2(\dive\ue)\ue$ which has been added as a correction  to avoid the paradox of increasing the kinetic energy along the motion. 

Since it will not affect our approximation process, for semplicity, from now on, we will take $\mu=1$ and $f^{\e}=0$.  

Furthermore we assign to the system \eqref{3.2} the following two initial conditions
\begin{equation*}
\ue(x,0)=\ue_{0}(x), \qquad \pe(x,0)=\pe_{0}(x).
\end{equation*}
It is worth to mention here that  the Navier Stokes equations require only one initial condition on the velocity $u$. Hence our approximation will be consistent if the initial datum on the pressure will be eliminated by an ``initial layer'' phenomenon which will be a consequence of the dispersive nature of the acoustic pressure waves. Since in the limit we have to deal with Leray solutions it is reasonable to require the finite energy constraint to be satisfied by the approximating sequences $(\ue,\pe)$. So we can deduce a natural behaviour to be imposed on the initial data $(\ue_{0},\pe_{0})$, namely
\begin{equation}
\tag {\bf{ID}}
\begin{split}
& \ue_{0}=\ue(\cdot, 0)\longrightarrow u_{0}=u(\cdot ,0)\ \text{strongly in}\  L^{2}(\Omega)\\ 
&\sqrt{\e}\pe_{0}=\sqrt\e \pe(\cdot,0)\longrightarrow 0\  \text{strongly in} \  L^{2}(\Omega).
\end{split}
\label{ID}
\end{equation}
Let us remark that  the convergence of $\sqrt{\e}\pe_{0}$ to $0$ is necessary to avoid  the presence of concentrations of energy.

Since we are in an exterior domain the system \eqref{3.2} needs to be supplemented with boundary data. 
Taking into account the Navier Stokes equations \eqref{3.1} the natural choice is to assign homogenous Dirichlet boundary condition to the velocity vector field $\ue$, namely
\begin{equation}
\ue|_{\partial\Omega}=0.
\tag{\bf{BC1}}
\label{B1}
\end{equation}
For $\pe$ the matter is more delicate.
Assuming that everything is smooth we would like to assume homogenous Dirichlet boundary condition for $\pe$. However if we consider the second equation of the system \eqref{3.2} we can rewrite the pressure as $\pe_{t}=-1/\e\dive \ue$, from the energy  priori estimate we only know that $\dive \ue$ uniformly  bounded in $L^{2}$, hence we don't have sufficient regularity in order to define a trace for $\pe$.
As we will see, in the Section \ref{secpres}, we are going to decompose the pressure in two parts. One connected with the viscosity part of the fluid and the other one with the convective terms. So, our acoustic pressure waves are given by the superposition of two different waves with different frequencies scales. These different scales provide different level of regularity for $\pe$, therefore on the boundary we have to assume
\begin{equation}
\Delta^{-\alpha}\pe|_{\partial\Omega}=0, \quad \text{for any $\alpha\geq \frac{1}{2}$}.
\tag{\bf{BC2}}
\label{B2}
\end{equation}
This condition seems to be in contrast with the boundary data usually associated to the pressure of the Navier Stokes equation, indeed, as it is well known, if one considers the Navier Stokes equation the natural behaviour for the pressure at the boundary is of Neumann type.  Chorin in  \cite{Ch67} dealt with the same issue. In \cite{Ch67} he considered  the artificial compressibility approximation method in the $2-D$ case. Let us denote the space coordinates by $(x_{1}, x_{2})$ and the boundary given by the line $x_{2}=0$, moreover denote by
$$u^{n}_{m(i,j)}=u(i\Delta x_{1}, j\Delta x_{2}, n\Delta t) \qquad p_{i,j}^{n}=p(i\Delta x_{1}, j\Delta x_{2}, n\Delta t)$$
the approximation of the sequences $\ue$ and $\pe$. In  \cite{Ch67}  Chorin assigns Dirichlet boundary condition to $p_{i,j}^{n}$ in the following way, namely if the boundary line $x_{2}=0$ is represented by $j=1$,  then  taking into account the equation $\e\pe_{t}+\dive \ue=0$, he writes $p_{i,1}^{n+1}$:
$$\e p_{i,1}^{n+1}-\e p_{i,1}^{n-1}=-2\frac{\Delta t}{\Delta x_{2}}(u^{n}_{2(i,2)}-u^{n}_{2(i,1)})-\frac{\Delta t}{\Delta x_{1}}(u^{n}_{1(i+1,1)}-u^{n}_{1(i-1,1)}),$$
where $u^{n}_{m(i,j)}$ are known form the previous step.\\
The issue of the correct boundary condition to assign to the pressure was also studied by many other authors. Among them Gresho and Sani \cite{GS87}, (see also \cite{GSSP06}, \cite{AD88}) showed that Dirichlet and Neumann boundary conditions for the pressure give the same solution. In fact they recover appropriate Dirichlet boundary conditions for the pressure by taking into account the normal derivative  of the pressure associated with the Neumann boundary condition and by using Green's functions. 
To be more precise, in our case we can say that in the limit the pressure plays the role  of the Lagrange multiplier associated with the constrain $\dive u=0$. In this regard, in analogy with problems of motion  of constrained  rigid bodies, the behaviour of the pressure must be deduced in terms of the velocity field $u$. 
\begin{remark}
For semplicity we assume homogenous boundary conditions \eqref{B1} and \eqref{B2}.  If we are in the nonhomogeneous case, let say for example $\ue|_{\partial\Omega}=u_{\Gamma}$, by changing $\ue$ into $\ue-u_{\Gamma}$, as usual,  we come into the case of zero Dirichlet boundary conditions.
\end{remark}


Now we can state our main result. The convergence of $\{\ue\}$ will be  described by analyzing the convergence of the associated Hodge decomposition.
\begin{theorem}
Let $(\ue,\pe)$ be a sequence of weak solution  of the system \eqref{3.2}, assume that the initial data satisfy \eqref{ID} and the boundary conditions \eqref{B1} and \eqref{B2}  hold. Then 
\begin{itemize}
  \item [\bf{(i)}] There exists $u\in L^{\infty}([0,T];L^{2}(\Omega))\cap L^{2}([0,T];\dot H^{1}(\Omega))$ such that 
  \begin{equation*}
\ue\rightharpoonup u \quad \text{weakly in $L^{2}([0,T];\dot H^{1}(\Omega))$}.
\end{equation*}
  \item [\bf{(ii)}] The gradient component $Q\ue$ of the vector field $\ue$ satisfies
  \begin{equation*}
Q\ue\longrightarrow 0\quad \text{ strongly in $L^{2}([0,T];L^{p}(\Omega))$, for any $p\in [4,6)$}.
\end{equation*}
 \item [\bf{(iii)}] The divergence free component $P\ue$ of the vector field $\ue$ satisfies
   \begin{equation*}
P\ue\longrightarrow Pu=u\quad \text{strongly  in $L^{2}([0,T];L^{2}_{loc}(\Omega))$}.
\end{equation*}
 \item [\bf{(iv)}] The sequence $\{\pe\}$ will converge in the sense of distribution  to 
 \begin{equation*}
p=\Delta^{-1}div \left((u\cdot\nabla)u\right)=\Delta^{-1}tr((Du)^{2}).
\end{equation*}
\item [\bf{(v)}] $u=Pu$ is a Leray weak solution to the incompressible Navier Stokes equation
\begin{align*}
&P(\partial_{t} u-\Delta u+(u\cdot\nabla)u)=0,\\
&u(x,0)=u_{0}(x),\qquad u|_{\partial\Omega}=0.
\end{align*}
\item [\bf{(vi)}]  The following energy inequality holds for all $t\in[0,T]$,
\begin{equation}
\frac{1}{2}\int_{\Omega}|u(x,t)|^{2}dx+\int_{0}^{t}\!\!\int_{\Omega}|\nabla u(x,t)|^{2}dxdt\leq 
\frac{1}{2}\int_{\Omega}|u_{0}(x)|^{2}dx.
\label{en}
\end{equation}
\end{itemize}
\label{tM}
\end{theorem}
\begin{remark}
Notice that in the limit we can recover the Dirichlet boundary condition for $u$ since the approximating sequence $\ue\in L^{2}([0,T];\dot H^{1}(\Omega))$ and the trace operator is continuos.  
\end{remark}
\begin{remark}
This theorem can be easily extended to the nonhomogeneous equation \eqref{3.2}, by assuming 
\begin{equation*}
f^{\e}\longrightarrow f\qquad \text{strongly in $L^{2}([0,T];H^{-1}(\Omega))$}.
\end{equation*}
\end{remark}

\section{Energy estimates}
In this section we wish to establish the a priori estimates, independent on $\e$, for the solutions of the system \eqref{3.2} which are necessary to prove the Theorem \ref{tM}.
In particular we will recover the  a priori estimates that come from the classical energy estimates related to the system \eqref{3.2}. 
\begin{proposition}
Let us consider the solution $(\ue, \pe)$ of the Cauchy problem for the system \eqref{3.2}. Assume that the hypotheses \eqref{ID}  and the condition \eqref{B1} hold, then one has
\begin{equation}
\label{4.1}
E(t)+\int_{0}^{t}\!\!\int_{\Omega}|\nabla \ue(x,s)|^{2}dxds=E(0),
\end{equation}
where we set
\begin{equation}
\label{4.2}
E(t)=\int_{\Omega}\left( \frac{1}{2}|\ue(x,t)|^{2}+ \frac{\e}{2} |\pe(x,t)|^{2}\right)dx.
\end{equation}
Moreover,
\begin{align} 
& \sqrt{\e}\pe &\quad & \text{is bounded in $L^{\infty}([0,T];L^{2}(\Omega))$,} \label{4.3}\\
& \e\pe_{t} &\quad & \text{is relatively compact in $H^{-1}([0,T]\times \Omega),$}  \label{4.4}\\
& \nabla\ue &\quad & \text{is bounded in $L^{2}([0,T]\times\Omega),$}  \label{4.5}\\
& \ue &\quad & \text{is bounded in $L^{\infty}([0,T];L^{2}(\Omega))\cap L^{2}([0,T];L^{6}(\Omega)),$}  \label{4.6}\\
(&\ue \!\cdot\!\nabla)\ue &\quad & \text{is bounded in $L^{2}([0,T];L^{1}(\Omega))\cap L^{1}([0,T];L^{3/2}(\Omega)),$}  \label{4.7}\\
(& \dive\ue)\ue &\quad & \text{is bounded in $L^{2}([0,T];L^{1}(\Omega))\cap L^{1}([0,T];L^{3/2}(\Omega)).$}\label{4.8}
\end{align}
\label{p4.1}
\end{proposition}
\begin{proof}
We multiply, as usual,  the first equation of the system \eqref{3.2} by $\ue$ and the second by $\pe$, then we sum up and integrate by parts in space and time, hence  we get \eqref{4.1}. The estimates \eqref{4.3},  \eqref{4.4}, \eqref{4.5} follow from \eqref{4.1}, while \eqref{4.6} follows from \eqref{4.1} and Sobolev embeddings theorems. Finally \eqref{4.7} and \eqref{4.8} come from \eqref{4.5} and \eqref{4.6}.
\end{proof}
\begin{remark}
We want to point out as for  the Navier Stokes equations, here, in order to get the estimate \eqref{4.1} we only used the boundary condition on the velocity vector field $\ue$, no boundary conditions are required for the pressure $\pe$.
\end{remark}

\section{Acoustic pressure wave equation}
\label{secpres}

In order to perform our limiting process we need some more estimates that don't follow from the previous one. In fact from the estimates \eqref{4.5} and \eqref{4.6} we only get the weak convergence of $\ue$.  Here we want to exploit  the dispersive behaviour of the pressure acoustic wave in order to damp its disturbing effect. Let us differentiate with respect to time the equation $\eqref{3.2}_{2}$, by using $\eqref{3.2}_{1}$, we get that $\pe$ satisfies the following wave equation
\begin{equation}
\e\partial_{tt}\pe-\Delta \pe +\Delta \dive\ue-\dive\left(\left(\ue\cdot\nabla\right)\ue +\frac{1}{2}(\dive \ue)\ue \right)=0.
\label{5.1}
\end{equation}
Now we rescale the time variable, the velocity and the pressure in the following way 
\begin{equation}
\label{5.5}
\tau=\frac{t}{\sqrt{\e}}, \quad \ut(x,\tau)=\ue(x,\sqrt{\e}\tau), \quad \pt(x,\tau)=\pe(x,\sqrt{\e}\tau).
\end{equation}
As a consequence of this scaling the equation \eqref{5.1} becomes
\begin{equation}
\label{5.6}
\partial_{\tau\tau}\pt-\Delta\pt =-\Delta \dive \ut+\dive\left(\left(\ut\cdot\nabla\right)\ut +\frac{1}{2}(\dive \ut)\ut\right).
\end{equation}
Now, taking into account the right hand-side of \eqref{5.6} we decompose $\pt$ as the sum of two component $\pt_{1}$ and $\pt_{2}$. In particular we have that $\pt_{1}$ is related to the viscosity of the fluid and satisfies
\begin{equation}
\label{5.7}
\begin{cases}
     \partial_{\tau\tau}\Delta^{-1}\pt_{1}-\Delta\Delta^{-1}\pt_{1} =-\Delta^{-1}\Delta\dive\ut= F_{1} \\
    \Delta^{-1}\pt_{1}(x,0)=\partial_{\tau} \Delta^{-1}\pt_{1}(x,0)=0,\\
    \Delta^{-1} \pt_{1}|_{\partial \Omega}=0,
        \end{cases}
\end{equation}    
While $\pt_{2}$ is connected with the convective part of the fluid and  verifies the  following wave equation:
\begin{equation}
\label{5.8} 
\begin{cases}   
\displaystyle{ \partial_{\tau\tau}\Delta^{-1/2}\pt_{2}-\Delta\Delta^{-1/2}\pt_{2} =\Delta^{-1/2}\dive\left(\left(\ut\cdot\nabla\right)\ut +\frac{1}{2}(\dive \ut)\ut \right)=F_{2}}\\
\Delta^{-1/2}\pt_{2}(x,0)=\Delta^{-1/2}\pt(x,0)  \quad  \partial_{\tau} \Delta^{-1/2}\pt_{2}(x,0)=\partial_{\tau} \Delta^{-1/2}\pt(x,0),\\
 \Delta^{-1/2}\pt_{2}|_{\partial \Omega}=0.
 \end{cases}
 \end{equation}
Therefore we are able to prove the following theorem.
\begin{theorem}
Let us consider the solution $(\ue, \pe)$ of the Cauchy problem for the system \eqref{3.2}. Assume that the hypotheses \eqref{ID} and the boundary conditions \eqref{B1}, \eqref{B2} hold. Then we set the following estimate 
\begin{align}
\hspace{-1mm}\e^{3/8}\|\pe\|_{L^{4}_{t} W^{-2,4}_{x}}+\e^{7/8}\|\partial_{t}\pe\|_{L^{4}_{t} W^{-3,4}_{x}}&\lesssim \sqrt{\e}\|\pe_{0}\|_{L^{2}_{x}}+\|\dive\ue_{0}\|_{\dot H^{-1}_{D}}\notag\\
&+\|\left(\ue\cdot\nabla\right)\ue +\frac{1}{2}(\dive \ue)\ue\|_{L^{1}_{t}L^{3/2}_{x}}\notag\\
&+\sqrt{T}\|\dive \ue\|_{L^{2}_{t}L^{2}_{x}}\label{5.9}.
\end{align}
\label{t5.1.1}
\end{theorem}
\begin{proof}
Since $\pt_{1}$ and $\pt_{2}$ are solutions of the wave equations \eqref{5.7}, \eqref{5.8}, we can apply the Strichartz estimates \eqref{s1} and \eqref{s2}, with $(x,\tau)\in \Omega\times\left (0,T/\sqrt \e\right)$.
First of all we use the Strichartz estimate \eqref{s1} with $w=\Delta^{-1}\pt_{1}$ and we get
\begin{equation}
\label{5.10}
\|\Delta^{-1}\pt_{1}\|_{L^{4}_{\tau,x}}+\|\partial_{\tau}\Delta^{-1}\pt_{1}\|_{L^{4}_{\tau} W^{-1,4}_{x}}\lesssim 
\|F_{1}\|_{L^{1}_{\tau}L^{2}_{x},}
\end{equation}
namely
\begin{equation}
\|\pt_{1}\|_{L^{4}_{\tau} W^{-2,4}_{x}}+\|\partial_{\tau}\pt_{1}\|_{L^{4}_{\tau} W^{-3,4}_{x}}\lesssim \frac{\sqrt{T}}{\e^{1/4}}\|\dive \ut\|_{L^{2}_{\tau}L^{2}_{x}}.
\label{5.11}
\end{equation}
In the same way we apply  the estimate \eqref{s2}  with $w=\Delta^{-1/2}\pt_{2}$ and we obtain
\begin{align}
\label{5.12}
\|\Delta^{-1/2}\pt_{2}\|_{L^{4}_{\tau,x}}+\|\partial_{\tau}\Delta^{-1/2}\pt_{2}\|_{L^{4}_{\tau} W^{-1,4}}&\lesssim 
\|\Delta^{-1/2}\partial_{\tau}\pt(x,0)\|_{\dot H^{-1/2}_{D}}+\|F_{2}\|_{L^{1}_{\tau}L^{3/2}_{x},}
\end{align}
namely,
\begin{align}
\label{5.13}
\|\pt_{2}\|_{L^{4}_{\tau} W^{-1,4}_{x}}+\|\partial_{\tau}\pt_{2}\|_{L^{4}_{\tau} W^{-2,4}_{x}}&\lesssim 
\|\pt(x,0)\|_{\dot H^{-1/2}_{D}}+\|\partial_{\tau}\pt(x,0)\|_{\dot H^{-3/2}_{D}}\notag
\\&+\|\left(\ut\cdot\nabla\right)\ut +\frac{1}{2}(\dive \ut)\ut \|_{L^{1}_{\tau}L^{3/2}_{x}.}
\end{align}
Now by taking into account  \eqref{5.11}, \eqref{5.13} it follows that $\pt$ verifies
\begin{align}
\|\pt\|_{L^{4}_{\tau} W^{-2,4}_{x}}+\|\partial_{\tau}\pt\|_{L^{4}_{\tau} W^{-3,4}_{x}}&\leq \|\pt_{1}\|_{L^{4}_{\tau} W^{-2,4}_{x}}+\|\pt_{2}\|_{L^{4}_{\tau} W^{-1,4}_{x}}\\&+
\|\partial_{\tau}\pt_{1}\|_{L^{4}_{\tau} W^{-3,4}_{x}}+\|\partial_{\tau}\pt_{2}\|_{L^{4}_{\tau} W^{-2,4}_{x}}
\notag
\\&\lesssim 
\|\pt(x,0)\|_{\dot H^{-1/2}_{D}}+\|\partial_{\tau}\pt(x,0)\|_{\dot H^{-3/2}_{D}}\notag\\&+
\|\left(\ut\cdot\nabla\right)\ut +\frac{1}{2}(\dive \ut)\ut \|_{L^{1}_{\tau}L^{3/2}_{x}}\notag\\
&+\frac{\sqrt{T}}{\e^{1/4}}\|\dive \ut\|_{L^{2}_{\tau}L^{2}_{x}}.
\label{5.14}
\end{align}
Using backwards the scaling \eqref{5.5} we end up with \eqref{5.9}.
\end{proof}
\begin{remark}
Notice that the right hand-side of \eqref{5.9} is uniformly bounded in $\e$ by means of $E(0)$, see Proposition \ref{p4.1}.
\end{remark}
\section{Strong convergence}
This section is devoted to the proof of the strong convergence of $Q\ue$ and $P\ue$. In particular we will show that the gradient part of the velocity $Q\ue$ converges strongly to $0$, while the incompressible component of the velocity field $P\ue$ converges strongly to $Pu=u$, where $u$ is the limit profile as $\e\downarrow 0$ of $\ue$. We start this section with some easy consequences of the a priori estimates established in the previous section.
\begin{corollary}
\label{c6.1}
Let us consider the solution $(\ue, \pe)$ of the Cauchy problem for the system \eqref{3.2}. Assume that the hypotheses \eqref{ID} and the conditions \eqref{B1}, \eqref{B2} hold. Then, as $\e\downarrow 0$, one has
\begin{align}
&\e\pe\longrightarrow 0 &\quad& \text{strongly in $L^{\infty}([0,T];L^{2}(\Omega))\cap L^{4}([0,T];W^{-2,4}(\Omega))$,}\label{6.1}\\
&\dive \ue \longrightarrow 0 &\quad& \text{strongly in $ W^{-1,\infty}([0,T];L^{2}(\Omega))\cap L^{4}([0,T];W^{-3,4}(\Omega))$}.\label{6.2}
\end{align}
\end{corollary}
\begin{proof}
\eqref{6.1}, \eqref{6.2} follow  from the estimates \eqref{4.3}, \eqref{5.9} and the second equation of the system \eqref{3.2}. 
\end{proof}

\subsection{Strong convergence of $Q\ue$}
Here, we wish to show that the gradient part of the velocity field $Q\ue$ goes strongly to $0$ as $\e\downarrow 0$. This will be a consequence of the estimate \eqref{5.9} provided that we observe that by using the second equation of \eqref{3.2} we can rewrite $Q\ue$ as 
$$Q\ue=\nabla\Delta_{N}^{-1}\dive\ue=-\e\nabla\Delta_{N}^{-1}\partial_{t}\pe.$$
\begin{proposition}
Let us consider the solution $(\ue, \pe)$ of the Cauchy problem for the system \eqref{3.2}. Assume that the hypotheses \eqref{ID} and the  conditions \eqref{B1}, \eqref{B2} hold. Then  as $\e\downarrow 0$,
\begin{equation}
Q\ue \longrightarrow 0 \quad \text{strongly in $ L^{2}([0,T];L^{p}(\Omega))$ for any $p\in [4,6)$ }.
\label{6.1.1}
\end{equation}
\label{p6.1.1}
\end{proposition}
\begin{proof}
In order to prove the Proposition \ref{p6.1.1} we split $Q\ue$ as follows
\begin{equation*}
\|Q\ue\|_{L^{2}_{t}L^{p}_{x}}\leq \|Q\ue-Q\ue\ast j_{\alpha}\|_{L^{2}_{t}L^{p}_{x}}+\|Q\ue\ast j_{\alpha}\|_{L^{2}_{t}L^{p}_{x}}=J_{1}+J_{2},
\end{equation*}
where $j_{\alpha}$ is the smoothing kernel defined in Lemma \ref{ly}.
Now we estimate separately $J_{1}$ and $J_{2}$. For $J_{1}$ by using \eqref{y1} we get
\begin{equation}
\label{6.1.2}
J_{1}\leq \alpha^{1-3\left(\frac{1}{2}-\frac{1}{p}\right)}\left(\int_{0}^{T}\|Q\nabla \ue(t)\|_{L^{2}_{x}}^{2} dt\right)\leq \alpha^{1-3\left(\frac{1}{2}-\frac{1}{p}\right)}\|Q\nabla \ue\|_{L^{2}_{t}L^{2}_{x}}.
\end{equation}
Hence from the identity $Q\ue=-\e^{1/8}\nabla\Delta^{-1}_{N}\e^{7/8}\partial_{t}\pe$ and by the inequality \eqref{y2} we get  $J_{2}$ satisfies the following estimate
\begin{align}
J_{2}&\leq \e^{1/8}\|\nabla\Delta^{-1}_{N}\e^{7/8}\partial_{t}\pe\ast\psi\|_{L^{2}_{t}L^{p}_{x}}
\leq \e^{1/8}\alpha^{-2-3\left(\frac{1}{4}-\frac{1}{p}\right)}\|\e^{7/8}\partial_{t}\pe\|_{L^{2}_{t}W^{-3,4}_{x}}\notag\\
&\leq \e^{1/8}\alpha^{-2-3\left(\frac{1}{4}-\frac{1}{p}\right)}T^{1/4}\|\e^{7/8}\partial_{t}\pe\|_{L^{4}_{t}W^{-3,4}_{x}}.
\label{6.1.3}
\end{align}
Therefore,  summing up \eqref{6.1.2} and \eqref{6.1.3} and by using \eqref{4.5} and \eqref{5.9} and remembering that $Q$ is a bounded operator from $L^{2}$ into $L^{2}$, we conclude  for any $p\in [4,6)$ that
\begin{equation}
\|Q\ue\|_{L^{2}_{t}L^{p}_{x}}\leq C\alpha^{1-3\left(\frac{1}{2}-\frac{1}{p}\right)}+C_{T}\e^{1/8}\alpha^{-2-3\left(\frac{1}{4}-\frac{1}{p}\right)}.
\label{6.1.4}
\end{equation}
Finally, we choose $\alpha$ in terms of $\e$ in order that the two terms in the right hand side of the previous inequality have the same order, namely
\begin{equation}
\alpha=\e^{1/18}.
\end{equation}
Therefore we obtain
\begin{equation*}
\displaystyle{\|Q\ue\|_{L^{2}_{t}L^{p}_{x}}\leq C_{T}\e^{ \frac{6-p}{36p}}\quad \text{for any $p\in [4,6)$.}}
\end{equation*}
\end{proof}

\subsection{Strong convergence of $P\ue$}
It remains to prove the strong compactness of the incompressible component of the velocity field.  To achieve this goal we need to prove   some time regularity properties of $P\ue$.
\begin{theorem}
\label{t6.2.1}
Let us consider the solution $(\ue, \pe)$ of the Cauchy problem for the system \eqref{3.2}. Assume that the hypotheses \eqref{ID} and  conditions \eqref{B1}, \eqref{B2} hold. Then  as $\e\downarrow 0$
\begin{equation}
P\ue \longrightarrow Pu, \qquad \text{strongly in $L^{2}(0,T;L^{2}_{loc}(\Omega))$}.
\label{6.2.7}
\end{equation}
\end{theorem}
\begin{proof}
From the Proposition \ref{p4.1} we know that $P\ue$ is uniformly bounded in $L^{2}_{t}\dot H^{1}_{x}$. The strong convergence \eqref{6.2.7} follows by applying the Theorem \ref{LA} provided  that for all $h\in (0,1)$ we have  
\begin{equation}
\label{6.2.1}
\|P\ue(t+h)-P\ue(t)\|_{L^{2}([0,T]\times \Omega)}\leq C_{T}h^{1/5}.
\end{equation}
Let us set $z^{\e}=\ue(t+h)-\ue(t)$, then we have
\begin{align}
\hspace{-0,25 cm}\|P\ue(t+h)-P\ue(t)\|^{2}_{L^{2}([0,T]\times \Omega)}&=\int_{0}^{T}\!\!\int_{\Omega}dtdx(Pz^{\e})\cdot(Pz^{\e}-Pz^{\e}\ast j_{\alpha})\notag\\&+\int_{0}^{T}\!\!\int_{\Omega}dtdx(Pz^{\e})\cdot(Pz^{\e}\ast j_{\alpha})\notag\\
&=I_{1}+I_{2}.
\label{6.2.2}
\end{align}
By using \eqref{y1} we can estimate $I_{1}$ in the following way
\begin{align}
I_{1}&\leq \|Pz^{\e}\|_{L^{\infty}_{t}L^{2}_{x}}\int_{0}^{T}\|Pz^{\e}(t)-(Pz^{\e}\ast j_{\alpha})(t)\|_{L^{2}_{x}}dt\notag\\&\lesssim \alpha T^{1/2}\|\ue\|_{L^{\infty}_{t}L^{2}_{x}}\|\nabla P\ue\|_{L^{2}_{t,x}}.
\label{6.2.3}
\end{align}
Let us reformulate $Pz^{\e}$ in integral form by using the equation $\eqref{3.2}_{1}$, hence
\begin{align}
\hspace{-0.3cm}I_{2}\leq\left|\int_{0}^{T}\!\!\!dt\!\!\int_{\Omega}\!\!\!dx \!\!\int_{t}^{t+h}\!\!\!ds(\Delta \ue-\left(\ue\cdot\nabla\right)\ue -\frac{1}{2}\ue(\dive \ue)(s,x)\cdot (Pz^{\e}\ast j_{\alpha})(t,x)\right|.
\label{6.2.4}
\end{align}
Then integrating by parts and by using \eqref{y2}, we deduce
\begin{align}
I_{2}&\leq hT^{1/2}\|\ue\|_{L^{\infty}_{t}L^{2}_{x}} \|\nabla\ue\|^{2}_{L^{2}_{t,x}}\notag\\
&+C\alpha^{-3/2}T^{1/2}\|\ue\|_{L^{\infty}_{t}L^{2}_{x}}\left(\!h\!\int_{t}^{t+h}\!\!\!\|\left(\ue\cdot\nabla\right)\ue -\frac{1}{2}(\dive \ue)\ue\|^{2}_{L^{1}_{x}}ds\right)^{1/2}\notag\\
&\leq hT^{1/2}\|\ue\|_{L^{\infty}_{t}L^{2}_{x}}\left(\|\nabla\ue\|_{L^{2}_{t,x}}+C\alpha^{-3/2}\|\left(\ue\cdot\nabla\right)\ue -\frac{1}{2}(\dive \ue)\ue\|_{L^{2}_{t}L^{1}_{x}}\right).
\label{6.2.5}
\end{align}
Summing up $I_{1}$, $I_{2}$ and by taking into account \eqref{4.5}--\eqref{4.8}, we have
\begin{equation}
\|P\ue(t+h)-P\ue(t)\|^{2}_{L^{2}([0,T]\times \Omega)}\leq CT^{1/2}(\alpha +h\alpha^{-3/2}+h),
\label{6.2.6}
\end{equation}
by choosing $\alpha=h^{2/5}$, we end up with \eqref{6.2.1}.
\end{proof}

\section{Proof of the Theorem \ref{tM}}
\begin{itemize}
\item[{\bf (i)}] It follows from the estimate \eqref{4.6}.
\item[{\bf (ii)}] It is a consequence of  the Proposition \ref{p6.1.1}.
\item[{\bf (iii)}] By taking into account the decomposition $\ue=P\ue+Q\ue$, the Theorem \ref{t6.2.1} and the Proposition \ref{p6.1.1} we have that
\begin{equation}
\label{7.1}
P\ue\longrightarrow u \qquad \text{strongly in $L^{2}([0,T];L^{2}_{loc}(\Omega))$.}
\end{equation}
\item[{\bf (iv)}]  Let us apply the Leray projector $Q$ to the equation $\eqref{3.2}_{1}$, then it follows
\begin{equation}
\label{7.2}
\nabla \pe =Q\Delta \ue- Q\left((\ue\cdot\nabla)\ue) +\frac{1}{2} \ue \dive Q\ue\right).
\end{equation}
Now by choosing a test function $\varphi \in C^{\infty}_{0}(\Omega\times[0,T])$ and by taking into account \eqref{4.5}, \eqref{6.1.1}, \eqref{6.2.7} and \eqref{7.1}, we get, as $\e \downarrow 0$, 
\begin{align}
\langle\ue \dive Q\ue, Q\varphi\rangle&\leq\|Q\ue\|_{L^{2}_{t}L^{4}_{x}} \|\nabla\ue\|_{L^{2}_{t}L^{2}_{x}}\|Q\varphi\|_{L^{\infty}_{t}L^{4}_{x}} \notag\\&+ \|Q\ue\|_{L^{2}_{t}L^{4}_{x}} \|\ue\|_{L^{\infty}_{t}L^{2}_{x}}\|\nabla Q\varphi\|_{L^{2}_{t}L^{4}_{x}} \rightarrow 0, 
\end{align}
\begin{align}
\langle Q((\ue\cdot\nabla)\ue),\varphi\rangle&=\langle( P\ue \cdot\nabla) P\ue),Q\varphi\rangle+\langle (Q\ue\cdot\nabla)Q\ue,Q\varphi\rangle\notag\\&\leq\langle (P\ue\cdot\nabla)  P\ue,Q\varphi\rangle\notag+\|Q\ue\|_{L^{2}_{t}L^{4}_{x}} \|\nabla Q\ue\|_{L^{2}_{t,x}}\|\nabla Q\varphi\|_{L^{\infty}_{t}L^{4}_{x}}\notag\\&\rightarrow \langle (u\cdot\nabla)u),Q\varphi\rangle=\langle Q((u\cdot\nabla)u) , \varphi\rangle,
\end{align}
\begin{align}
\langle Q\Delta\ue, \varphi \rangle&=\langle P\ue,  \Delta Q\varphi \rangle +\langle Q\ue, \Delta Q\varphi \rangle \notag\\
&\leq \langle P\ue, \Delta Q\varphi \rangle+\|Q\ue\|_{L^{2}_{t}L^{4}_{x}}\|\Delta Q\varphi\|_{L^{2}_{t}L^{4/3}_{x}}\rightarrow 0,
\end{align}
where in the previous computation we took into account that $Q$ is a continuos operator but doesn't commute with derivatives.
So  as $\e \downarrow 0$ we have, 
\begin{equation}
\label{7.3}
\langle \nabla \pe , \varphi \rangle \longrightarrow\langle \nabla\Delta^{-1}_{N}\dive((u\cdot\nabla)u) , \varphi\rangle. 
\end{equation}
\item[{\bf (v)}]In a similar way we can pass  into the limit inside the system \eqref{3.2} and we get  $u$ satisfies the following equation in $\mathcal{D}'([0,T]\times \Omega)$
\begin{align}
\label{7.4}
&P(\partial_{t} u-\Delta u+(u\cdot\nabla)u)=0,\notag\\
&u(x,0)=u_{0}(x),\qquad u|_{\partial\Omega}=0,
\end{align}
where we used (ID) and (BC1), (BC2) and the fact that the trace operator is bounded (see \cite{Gag57}, \cite{Ada75})
\item[{\bf (vi)}] Finally we prove the energy inequality. By using the weak lower semicontinuity of the weak limits, the hypotheses (ID) and denoting 
by $\chi$ the weak-limit of $\sqrt{\e}\pe $,  we have
\begin{align*}
&\int_{\Omega}\frac{1}{2}|\chi|^{2}dx+\int_{\Omega}\frac{1}{2}|u(x,t)|^{2}dx+\int_{0}^{t}\!\!\int_{\Omega}|\nabla u(x,t)|^{2}dxdt\notag\\&\leq
\liminf_{\e\to 0}\left(\int_{\Omega}\frac{1}{2}|\ue(x,t)|^{2}dx+\int_{\Omega}\frac{\e}{2}|\pe(x,t)|^{2}+\int_{0}^{t}\!\!\int_{\Omega}|\nabla \ue(x,t)|^{2}dxdt\right)\notag\\&=\liminf_{\e\to 0}\int_{\Omega}\frac{1}{2}\left(|\ue_{0}(x)|^{2}+\e|\pe_{0}(x)|^{2}\right)dx=\int_{\Omega}\frac{1}{2}|u_{0}(x)|^{2}dx.
\end{align*}
for any $t\in[0,T]$.
\end{itemize}
\bibliographystyle{amsplain}

\end{document}